\documentclass[a4paper]{amsart}
\usepackage{amsmath, amsfonts, amsthm}
\usepackage{amssymb,graphics,epsfig}
\usepackage{enumerate}

\newtheorem{defn}{Definition}
\newtheorem{thm}{Theorem}
\newtheorem{prop}{Proposition}
\newtheorem{cor}{Corollary}
\newtheorem{lemma}{Lemma}

\DeclareMathOperator{\diver}{\textrm{div}}

\newcommand{\DDD}{\mathcal{D}}
\newcommand{\LL}{\mathbb{L}}

\newcommand{\Lun}{L^1}
\newcommand{\Lunloc}{L^1_{loc}}
\newcommand{\Lde}{L^2}

\newcommand{\Lp}{L^p}

\newcommand{\RR}{\mathbb{R}}

\newcommand{\Com}{C^\infty_0}
\newcommand{\CCC}{C}
\newcommand{\Linf}{L^\infty}

\makeatletter \@addtoreset{equation}{section} \makeatother

\newcommand{\cqfd}{{\unskip\kern 6pt\penalty 500
\raise -2pt\hbox{\vrule\vbox to 6pt{\hrule width 6pt
\vfill\hrule}\vrule}\par}}

\title[2D Hamiltonian transport eq with $L^p$ coefficients.]{On Two-dimensional 
Hamiltonian Transport Equations with $L^p_{loc}$ 
coefficients}

\author{Maxime Hauray}

\address{M. Hauray: Universit\'e Paris-Dauphine, et CNRS}

\email{maxime.hauray@univ-amu.fr}

\begin{document}



\begin{abstract} We consider two-dimensional autonomous divergence free 
vector-fields in $\Lde_{loc}$. Under a condition on direction of the flow and on 
the set of critical points, we prove the existence and uniqueness of a stable 
a.e. flow and of renormalized solutions of the associated transport equation.
\end{abstract}

\maketitle

\section{Introduction}

We consider the following transport equation,
\begin{equation} \label{trans_dim}
\frac{\partial u}{\partial t}(t,x) + b(x) \cdot \nabla_x u(t,x) = 0
\end{equation}
with initial conditions
\begin{equation} \label{initial}
u(0,x)=u^o(x)
\end{equation}
where $t \in \RR$, $x \in \Omega$, $u^o: \Omega \rightarrow \RR$, $b:
\Omega \rightarrow \RR^2$ satisfies $div \, b =0$ and $u:\RR \times
\Omega \rightarrow \RR$. The domain $\Omega$ is the torus $\Pi^2$, or
$\RR^2$ but in that case we must assume that $b$ satisfies  some
natural growth conditions, or a bounded open regular subset of $\RR^2$
and $b$ is then required to be tangent to the surface $\partial
\Omega$. We assume that $u^0 \in \Lp$ for some $p \in [1,\infty]$.
 
As is well known, this transport equation is in some sense equivalent
to the ODE 
\begin{equation} \label{ODE_dim}
\dot{X}(t)= b(X(t))
\end{equation}

Let us begin with some definitions and a proposition in which we
always assume that $b$ belongs at least to $\Lun_{loc}$.

\begin{defn}
Given an initial condition in $\Linf$, a
solution of (\ref{trans_dim})-(\ref{initial}) is a function in
$\Linf([0,\infty) \times \Omega)$ satisfying for all $\phi \in \CCC^{\infty}_c
([0,\infty) \times \Omega)$
\begin{equation} \label{test}
\int_{[0,\infty) \times \Omega} u (\frac{\partial \phi}{\partial t} +
b \cdot \nabla_x \phi) = - \int_{\Omega} u^o \phi(0,\cdot)
\end{equation}
\end{defn}

\begin{defn}
We shall call renormalized solution a function $u$ in $\Lunloc([0,\infty) \times \Omega)$ such that $\beta(u)$ is a solution of (\ref{trans_dim}) with initial value $\beta(u^o)$, for all $\beta \in C^1_b(\RR)$, the set  of differentiable functions from $\RR$ to $\RR$ with bounded continuous derivative.
\end{defn}

{\bf Remark} In this definition, we do not ask $u$ to be a solution because if $u$ only belongs to $\Lunloc$, we cannot give a sense to the product $uv$. This is one of the reasons why we introduce this definition. But, this is of course an extension of the notion of solution. If $u \in \Linf$ is a renormalized solution, it may be shown using good $\beta$ that $u$ is a solution. 

We will give the next definition only for the case where $\Omega=\Pi^2$ or a bounded open subset of $\RR^2$ in order to simplify the presentation. We refer to \cite{DPL} for the adaptation to  the case of $\RR^n$ .

\begin{defn}
A flow defined almost everywhere (or a.e. flow) solving (\ref{ODE_dim})
is a function $X$ from $\RR \times \Omega$ to $\Omega$ satisfying
\begin{itemize}
\item[i.] $X \in \CCC(\RR, \Lun)^2$
\item[ii.] $\int_{\Omega} \phi(X(t,x))\,dx =\int_{\Omega} \phi(x)\,dx \quad \forall \phi \in \CCC^{\infty}, \quad \forall t \in \RR$ 
(preservation of the Lebesgue's measure)
\item[iii.] $X(s+t,x)=X(t,X(s,x))$ a.e. in x, $\forall s,t \in \RR$
\item[iv.] (\ref{ODE_dim}) is satisfied in the sense of distributions.
\end{itemize}
These properties implies that  
$$
\text{for a.e.} \; x \in \Omega, \;
\forall t \in \RR, \quad X(t,x)= x + \int_0^t b(X(s,x))\,ds.
$$ 
\end{defn}

Moreover, the useful following result is stated in \cite{Lio}.

\begin{prop}
The two following statements are equivalent
\begin{itemize}
\item[i.] For all initial condition $u^o \in \Lun$, there exists a unique stable renormalized solution of (\ref{trans_dim}).
\item[ii.] There exists a unique stable a.e. flow solution of (\ref{ODE_dim}).
\end{itemize}
Moreover the following condition (R) implies these two equivalent statements 
\begin{center}
(R) \hspace{1cm}  Every solution of (\ref{trans_dim}) belonging to $\Linf(\RR \times \Omega)$ is a renormalized solution.
\end{center}
\end{prop}

This method of resolution of  ODE's and associated transport equations was introduced by R.J. DiPerna and P.L. Lions in \cite{DPL}.  In this article, they show that if $b \in W^{1,1}_{loc}$, the problem (\ref{trans_dim})-(\ref{initial}) has a unique renormalized solution $u$. In fact, even if it is not stated in these terms in their article, we can adapt the method used in it to prove Proposition 1 and the fact that (R) is true when $b \in W^{1,1}_{loc}$. In our paper we will show that (R) holds provided that $b \in \Lde_{loc}$ and  that the following condition ($P_x$) on the local direction of $b$ is true for a sufficiently large set of points $x$ 
$$ 
(P_x) \qquad \exists \xi \in \RR^2,\  \alpha > 0,\  \epsilon > 0 \quad \text{such that for almost all}\quad y \in B(x,\epsilon) \quad  b(y) \cdot \xi \geq \alpha $$
This is a local condition and the quantities $\xi$, $\alpha$, $\epsilon$ depend on $x$.

We will also show that (R) still holds in the case of a physical Hamiltonian $H(x,y)= y^2/2 + V(x)$ with $V' \in \Lun_{loc}$.

This paper is a extension of L. Desvillettes and
F. Bouchut \cite{BD}, in which similar results are shown when $b$ is
continuous. The authors use the fact that since we have an
Hamiltonian, we can integrate the ODE to obtain a one dimensionnal
problem, that we are able to solve. We will adapt this method with
less regularity on $b$.

%%%%%%%%%%%%%%%%%%%%%%%%%%%%%%%%%%%%%%%%%%%%%%%%%%%%%%%%%
\section{Main result}
Since we are in dimension two and that $div(b)=0$, there exists a
scalar function $H$ (the hamiltonian) such that $\nabla H^{\perp} =
b$. If $b$ belongs to $\Lp$, then $H$ is in $W^{1,p}$. 

\begin{thm}
Let $\Omega'$ be an open subset of $\Omega$. Assume that $b \in
\Lde_{loc}(\Omega')$ and $(P_x)$ holds for every $x \in \Omega'$,
Then the condition (R) holds in $\Omega'$.
\end{thm}

\textbf{Remarks}
\begin{itemize}
\item[i.] Two is the critical exponent. It corresponds to the critical
  case $W^{1,1}$ in \cite{DPL} since in two dimension we have the Sobolev
  embedding from $W^{1,1}$ to $\Lde$. In the fourth paragraph, we shall
  describe a flow which is in $\Lp$ for all $p < 2$, which satisfy the
  condition $(P_x)$ everywhere but for which uniqueness is false.
\item[ii.] This theorem does not extend the result in \cite{DPL} in this
  particular case because a vector-fields in $W^{1,1}$ does not
  necessary satisfy the condition $(P_x)$. We can construct divergence
  free vector-fields in $W^{1,1}$ which does not satisfy the
  condition $(P_x)$ at any point $x$.  
\item[iii.] Our method allow to prove the existence and the uniqueness
  directly (i.e. without using (R)), but it raises many difficulties
  concerning localisation and the addition of critical points.
\item[iv.] Here we state a result for a subset of $\Omega$. Of course,
  a particular case of interest is the case when $\Omega'=\Omega$, where we
  may then use proposition 1 to obtain the existence and the uniqueness of
  an a.e. flow and of the solution of the transport equation. But the
  case $\Omega'\subsetneq \Omega$ will be useful when we will shall
  take into account some points where $(P_x)$ is not true.
\end{itemize}

\begin{proof}
We shall prove this result in several steps. First, we shall state and
prove some results about a change of variables. Then, we shall justify
its application in formula (\ref{trans_dim}), and obtain a new transport
equation. Finally, we reduces this problem to a one dimensionnal one,
that we are able to solve.

\textbf{Step 1.}  A change of variable.\\
It is sufficient to show the result stated in Theorem 1 locally. Then,
we shall work in a bounded neighbourhood $U$ of $x_0$, in which we
assume that  $b \cdot \xi > \alpha$ a.e. as in $(P_x)$. We define
$\Phi$   on $U$ by
$$ \Phi (x) = ( (x - x_0) \cdot \xi ,H(x))$$

  We wish to use $\Phi$ as a change of variable. For this, we use
  the following lemma

\begin{lemma}
  Assume that $H \in W^{1,p}(U)$ for $p \geq 2$, then there exist a
  bounded connected open set $V$ containing $(0,0)$ and $\Phi^{-1} \in
  W^{1,p}(V)$ such that
\begin{center}
\text{for almost all} $x \in U$, $\Phi(x) \in V$ and $\Phi^{-1} \circ
\Phi (x) = x$, \\
\text{for almost all} $y \in V$, $\Phi^{-1}(y) \in U$ and $\Phi \circ
\Phi^{-1} (y) = y$, \\
\end{center}

  $\Phi$ and $\Phi^{-1} $ leave invariant zero-measure sets. \\
  Moreover, we have for $f \in \Linf(V)$ the following formula:
\begin{equation} \label{cov}
\int_U f \circ \Phi (x) |D\Phi(x)| \, dx = \int_V f(y) \, dy
\end{equation}
\end{lemma}

\begin{figure}[b]
\centering
\includegraphics[width=\textwidth]{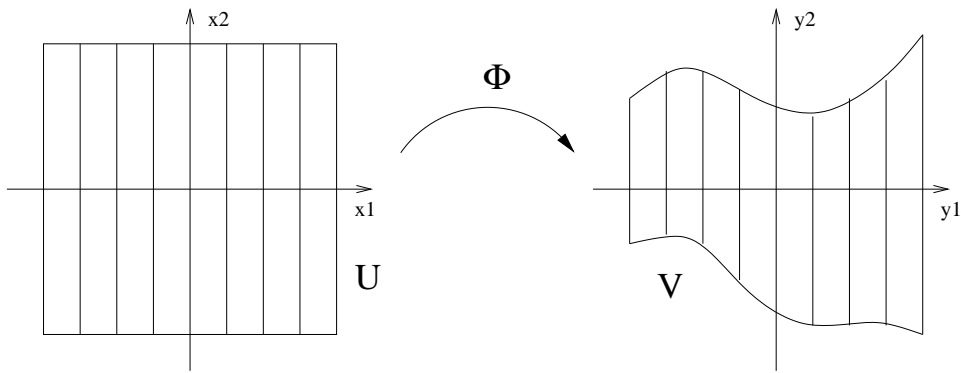}
\caption{ The $\Phi$ map}
\end{figure}

\begin{proof}[Proof of the lemma]
Without loss of generality, we may assume that $x_0=0$, $\xi=(-1,0)$,
$U=(-\eta, \eta)\times (-\eta, \eta)$. According to \cite{Zie} we can
assume, since $H$ is $W^{1,p}$, that $H$ is absolutely continuous on
almost all lines parallel to the coordinate axes and that it is true in
particular for the lines $\{y = \pm \eta \}$. Then we define a open
set $V$ by 
$$ V=\{(y_1,y_2) \in \RR^2 \mid H(y_1,-\eta) < y_2 < H(y_1,\eta)\} $$

 To show that $V$ is connected we have to show that $H(x_1,-\eta)
< H(x_1,\eta)$ for all $x_1 \in (-\eta, \eta)$. But we have $|b_1(x)| >
\alpha$ for almost all $x \in U$ then $H(x_1,\eta) - H(x_1,-\eta) >
2\eta \alpha$ for almost all $x_1 \in (-\eta, \eta)$, then for all
those $x_1$ by continuity.

$\Phi$ preserve the first coordinate, and for almost
all $x_1 \in (-\eta, \eta)$, $H(x_1, \cdot)$ is a strictly increasing
homeomorphism from $(-\eta, \eta)$ to $(H(x_1,-\eta),H(x_1,\eta))$. 
Hence we can define a suitable mesurable $\Phi^{-1}$. 

Now, we can prove the equation (\ref{cov}) using Fubini's
theorem. First we consider the case when $f$ is continuous. Then, we
have   
$$\int_U f \circ \Phi (x) |D\Phi(x)| \, dx = \int_{-\eta}^{\eta}
\left({\int_{-\eta}^{\eta} f(x_1,H(x_1,x_2))|b_1(x_1,x_2)| \, dx_2}\right) dx_1
$$

Next, if $F$ is $\CCC^1(\RR,\RR)$ and $\phi$ is in
$W^{1,p}([a,b])$, then $F \circ \phi$ is in $W^{1,p}([a,b])$ and $(F
\circ \phi)' = (F' \circ \phi)\phi'$. We now use this fact with $F$ a
primitive of $f$. Therefore, we can write 
$$ \int_{-\eta}^{\eta} f(x_1,H(x_1,x_2)) |b_1|(x_1,x_2) \, dx_2 =
\int_{H(x_1,-\eta)}^{H(x_1,\eta)} f(x_1,y) \, dy $$
And if we use Fubini's theorem again, we obtain the result.

Now, we prove (\ref{cov}) for an arbitrary function in $\Linf$. If $O$ is an
open subset of $V$, we choose a sequence of $f_n$ continuous such that
$f_n \rightarrow \chi_O$ everywhere when $n$ goes to $\infty$. By
increasing convergence, the result is still true for $\chi_O$. We have
it for the caracteristic function of an open set. If we use the fact
that $|b_1| \geq \alpha$, we obtain the inequality 
$$ \lambda(\Phi^{-1}(O)) \leq \frac{1}{\alpha} \lambda(O) $$
where $\lambda$ is the Lebesgue measure. Next, if $E$ is a zero measure
subset of $V$, we obtain (using the above inequality with open set of
small measure containing $E$) that $\Phi^{-1}(E)$ is also a
zero-measure set.

Now, if we approximate a $\Linf$-function $f$ by a sequence of
continuous functions $f_n$ converging to $f$ a.e., then the sequence
$f_n \circ \Phi$  converges to $f \circ \Phi$ a.e. and with the
dominated convergence theorem, we obtain the result for $f$.

The formula (\ref{cov}) may  be rewritten as follows
$$\int_U f(x) |D\Phi(x)| \, dx = \int_V f \circ \Phi^{-1} (y) \, dy
$$
By approximation, it  is always true provided the left hand side is
meaningful, as it is the case, for instance when $f$ belongs to $\LL^q(U)$,
with $q$ the conjugate exposant of $p$ ( $p^{-1}+q^{-1}=1$). And if $f
\in \LL^{a}(U)$, then $f \circ \Phi^{-1}$belongs to $\LL^b(V)$ with $b=a/q$.

To show that $\Phi^{-1}$ belongs to $W^{1,p}(V)$, and that
$D(\Phi^{-1})=(D\Phi)^{-1} \circ \Phi^{-1}$, the
most difficult case is to show that 
\begin{equation} \label{der}
\frac{\partial \Phi^{-1}_2}{\partial x_1} = -(\frac{b_2}{|b_1|}) \circ
\Phi^{-1} 
\end{equation}
First, since $b_2 \in \Lp(U)$ and $|b_1| > \alpha$, we can use the
change of variables to deduce
$$\int_V |\frac{b_2}{b_1}|^p \circ \Phi^{-1} = \int_U b_2^p b_1^{1-p}
$$
Hence, the right handside of (\ref{der}) belongs to $\Lp$.

Then, let $\phi$ be in $\CCC^{\infty}_o(V)$, we have
\begin{eqnarray*}
\int_V  \Phi_2^{-1} (y) \frac{\partial \phi}{\partial x_1} \, dy & =
& \int_U x_2 \frac{\partial \phi}{\partial x_1} \circ \Phi(x) \, |b_1(x)|\, dx
\\
 & = & \int_U x_2 \Big(\frac{\partial (\phi \circ \Phi)}{\partial x_1}
 (x) |b_1(x)| - \frac{\partial (\phi \circ \Phi)}{\partial
   x_2}(x) b_2(x) \Big) \, dx \\
 & = & \int_U \phi \circ \Phi(x) b_2(x) \, dx \\
 & = & \int_V \phi(y) \frac{b_2}{|b_1|} \circ \Phi^{-1}(x) \,dx \\
\end{eqnarray*}
and this is the expected result. To obtain the second line from the
first, we write  
$$ \partial_{x_1} (\phi \circ \Phi) = \partial_{x_1} \phi \circ \Phi -
b_2 \partial_{x_2} \phi \circ \Phi $$
$$ \partial_{x_2} (\phi \circ \Phi) = b_1 \partial_{x_2} \phi \circ
\Phi $$
And when $p \geq 2$, these two quantities belong to $\Lde$ and we may
multiply the first by $b_1$, the second by $-b_2$ and add them to
obtain the desired identity. To obtain the third line from the second,
we use an integration by parts and the fact that $div \, b =0$.
\end{proof}

\textbf{Step 2.} Equivalence with a new transport equation.

We now wish to apply the change of variables in the
formula (we recall that we assume that $\xi=(-1,0)$)
\begin{equation} \label{transint}
\int_{[0,\infty) \times U} u(\partial_t \phi + b.\nabla \phi) = -\int_U
u^o\phi^o 
\end{equation}

Since $u$ belongs to $\Linf(U)$, this expression make sense for
$\phi$ in $W^{1,q}_0 ([0,\infty) \times U)$ (here and below $q$ is always the
conjugate exponent of $p$). But, we want to apply (\ref{transint}) with
$\phi(t,x) = \psi(t,\Phi(y))$, where $\psi \in \Com([0,\infty)
\times V)$. In this case $\phi$ will belong to $W^{1,p} ([0,\infty)
\times U)$ and will also have a compact support because of the form of
$\Phi$. In addition, since  $p \geq 2$, we may write
 \begin{eqnarray*}
b \cdot \nabla \phi &=& b_1(\partial_{x_1} \psi \circ \Phi- b_2
\partial_{x_2} \psi \circ \Phi) + b_2
\, b_1 \partial_{x_2} \psi \circ \Phi\\
 &=& b_1 \partial_{x_1} \psi \circ \Phi
\end{eqnarray*}
and we obtain, denoting by $v(t,y)=u(t,\Phi^{-1}(y))$ and $J=|b_1|
\circ\Phi^{-1}$ 
$$ \int_{[0,\infty)\times V} v(\frac{1}{J(y)} \partial_t \psi(y)+
\partial_{x_1} \psi(y)) = -\int_V \frac{v^o \psi^o}{J} $$
for all $\psi$ in  $\Com([0,\infty) \times V)$. In other words, $v$ is
solution in $V$ (in the sense of the distributions) of
\begin{equation} \label{new}
\partial_t (\frac{v}{J}) + \partial_{x_1} v = 0
\end{equation}
with the initial condition $(v/J)(0,\cdot)=v^o/J$.

Conversely, if $v \in \Linf([0,\infty) \times V)$ is a solution of
(\ref{new}), we may test it against functions $\psi$ in $W^{1,1}_0
([0,\infty) \times V)$, and if $\phi$ is in $\Com([0,\infty) \times U)$ then
$\phi \circ \Phi^{-1}$ is in  $W^{1,1}_0 ([0,\infty) \times V)$. Thus we
may follow the above argument backwards, and we obtain that
(\ref{new}) is equivalent to (\ref{trans_dim}).

\textbf{Step 3.} Solution of the one dimensionnal problem.

  In view of the precedent steps, it is sufficient for us to show that (R) hold
for the equation (\ref{new}). But, in this equation there is no
derivative with respect to $y_2$. Therefore, it is equivalent to say that for
almost all $y_2$, $\partial_t (v/J) + \partial_{x_1} v = 0$ on the set
$\RR \times V_{y_2}$ with the good initial conditions (here
$V_{y_2}=\{y \in \RR \;|(y,y_2) \in V\}$). This would be obvious if $V$
were of the form $(a,b)\times (c,d)$, but we can always see $V$
as a countable union of such rectangular sets. And since an open
subset of $\RR$ is a countable union of open intervals, we just have
to show that the property (R) is true for the equation (\ref{new}) on
an interval $I=(a,b)$ of $\RR$, with $J \geq \alpha$ a.e. on $I$. 

Let $F$ be a primitive of $1/J$. F is continuous, strictly increasing
on $(a,b)$ onto $(F(a),F(b))$, and its inverse $F^{-1}$ belongs to
$W^{1,1}(F(a),F(b))$. Again, we may performe the change of
variables $y \mapsto z=F(y)$ and we obtain that
the equation (\ref{new}) on I is equivalent to 
\begin{equation} 
\partial_t w + \partial_{z} w = 0  \qquad \text{on} \quad [0,\infty)
\times (F(a),F(b))
\end{equation}
where $w(t,z)=v(t,F^{-1}(z))$. For this equation the property (R) is
true. In fact we have a flow  $ X(t,x)= F^{-1}(F(x)+t)$ for
(\ref{new}), but we need to be careful because we are not exactly on
the whole line and so this quantity is not defined for all~$t$. 
\end{proof}

%%%%%%%%%%%%%%%%%%%%%%%%%%%%%%%%%%%%%%%%%%%%%%%%%%%%
\section{Critical points}
In the preceding result, we assumed that the condition $(P_x)$
was true for all $x$. We want here to take into account possible
critical points. However, since we only assume that $b \in \Lun_{loc}$, we
cannot define critical points (of the Hamiltonian) as points where $b$ vanishes
(the usual notion when the flow is continuous). In some sense,
critical points mean for us all those points where $(P_x)$ is not
true. In fact, this yields a ``larger'' set of critical points. 

\subsection{Isolated critical points}
Our first result is the following

\begin{cor}
If $b$ satisfies $(P_x)$ everywhere in $\Omega$ except on a set of
isolated points, then the (R) hypothesis holds.
\end{cor}

\begin{proof}
Without loss of generality we may assume that $\Omega=\RR^2$, that
$(P_x)$ holds everywhere 
except at the origin $(0,0)$ and that $b \in \Lde$. We take $\psi \in
\Com(\RR)$ so that $\psi \equiv 1$ on a  neighbourhood of $(0,0)$ and
vanishes outside the ball $B_1$ of radius~$1$. We define
$\psi_\epsilon = \psi( \frac{\cdot}{\epsilon})$.\\ 
Let $\phi \in \Com([0,\infty) \times \RR^2)$, $\beta \in \CCC^1(\RR)$ and
$u$ be a solution of the transport equation on $\RR^2$, then $
(1-\psi_\epsilon) \phi \in \Com(\RR^2 \backslash \{(0,0)\}$, and since $u$
is a renormalized solution on $\RR^2 \backslash \{(0,0)\}$, we may write
\begin{equation}
\langle \partial_t \beta(u) + div(b \beta(u)), (1-\psi_\epsilon)
\phi \rangle = \int_{\RR^2} \beta (u^o) (1-\psi_\epsilon) \phi^o 
\end{equation}
\begin{multline}
\text{i.e.} \qquad \int_{[0,\infty) \times \RR^2} \beta(u)(1-\psi_\epsilon)
(\partial_t \phi + b \cdot \nabla \phi) - \int_{[0,\infty) \times
  \RR^2} \beta(u) \phi b \cdot \nabla \psi_\epsilon \\
= -\int_{\RR^2} \beta (u^o) (1-\psi_\epsilon) \phi^o
\end{multline}
When $\epsilon$ goes to $ 0$, the first integral converges to $\langle
\partial_t \beta(u) + div(b \beta(u)), \phi \rangle$, the
second one converges to 0 since
\begin{eqnarray*}
\left| \int_{\RR^2} \beta(u) \phi \nabla \psi_\epsilon \right| & \leq & C
\|b\|_{\Lde(B_\epsilon)} \|\nabla \psi_\epsilon\|_{\Lde} \\
 & \leq & C \|\nabla \psi\|_{\Lde}\|b\|_{\Lde(B_\epsilon)}
\end{eqnarray*}
and the right hand side converges to $-\int_{\RR^2} \beta (u^o)
\phi^o$. 

We conclude that  
$$\langle \partial_t \beta(u) + div(b
\beta(u)), \phi \rangle = \int_{\RR^2} \beta (u^o) \phi^o$$
 for all $\phi \in \Com([0,\infty) \times \RR^2)$. Hence $u$ is a renormalized
solution.
\end{proof}

%%%%%%%%%%%%%%%%%%
\subsection{A result with more regularity on $H$}
The above result, of course, does not allow for many critical
points. But we can allow much more with stronger conditions on $H$. First,
points where there exists a neighbourhood on which $b$ vanishes,
are obviously easy to handle. We shall call $O$ the set of all these
points, and $P$ the set of the points where $(P_x)$ is true. We denote
$Z = (O \cup P)^c$ (this complementary is taken in $\Omega$). It is
closed, since $O$ and $P$ are open . Then we have the following corolary. 

\begin{cor}
Assume that $H$ is continuous, $Z$ is a set of zero-measure in $\RR^2$
and $H(Z)$ is a set of zero-measure in $\RR$. Then (R) still holds.
\end{cor}

\textbf{Remarks}
\begin{itemize}
\item[i.] These conditions where introduced by L. Desvillettes and
  F. Bouchut in \cite{BD} in the case when $b$ is continuous. Here, we
  only rewrite their proof in a less regular case.
\item[ii.] If $p>2$, according to Sobolev embeddings, $H$ is automatically
  continuous. 
\item[iii.] We do not know if $H(Z)$ has zero-measure since we cannot
  apply Sard's lemma. 
\end{itemize}

\begin{proof}
Let $u$ be a solution of the transport equation (\ref{trans_dim}) in
$\Omega$, $\beta \in \CCC^1(\RR)$, $\phi$ a
$\CCC^{\infty}$-test function, and $K_o$ a compact set containing the
support of $\phi(t,\cdot)$ for all $t$. We define $Z_o=Z \cap K_o$ and $K=H(Z_o)$. Then $K$ is a
zero-measure compact set. Then, we can find functions $\chi_n \in
\Com(\RR)$ such that, $0 \leq \chi_n \leq 1$, $\chi_n \equiv 1$
on a neighbourhood of $K$ and $\chi_n \rightarrow \chi_K$, the
caracteristic function of $K$, when $n$ goes to $\infty$. We set
$\Psi_n= \chi_n \circ H$. $\Psi_n$ is continuous, belongs to
$W^{1,p}(\RR^2)$ and $\Psi_n \equiv 1$ on a neighbourhood of $Z_o$.

By theorem 1, $\beta(u)$ is a
solution of (\ref{trans_dim}) in $P$, and is also a solution in $O$, because
on this set $u$ is independent of the time. Since this  two sets are
open, $\beta(u)$ is a solution in $P \cup O$. $(1-\Psi_n)\phi$ is
continuous and belongs to $W^{1,p}([0,\infty) \times K_o)$ and has its support in
$[0,\infty) \times  (K_o \backslash Z_o)$. Hence, since $(K_o
\backslash Z_o) \subset P \cup O$ we can use it as a test fonction. We
have
$$ 
\langle \partial_t \beta(u) + div(b \beta(u)), (1-\Psi_n)
\phi \rangle = \int_{\Omega} \beta (u^o)(1-\Psi_n) \phi^o
$$
\begin{multline}
\text{i.e.} \qquad \int_{[0,\infty) \times \Omega} \beta(u)(1-\Psi_n)
  (\partial_t \phi + b \cdot \nabla \phi) - \int_{[0,\infty) \times
    \Omega} \beta(u)   b \cdot \nabla \Psi_n \\
= -\int_{\Omega} \beta (u^o)(1-\Psi_n)\phi^o 
\end{multline}
The second integral vanishes because $\nabla \Psi_n = (\Phi_n' \circ
H)\nabla H$ and $b=\nabla H ^{\perp}$.\\
The first integral converges by dominated convergence to 
$$ \int_{[0,\infty) \times H^{-1}(K)^c} \beta(u) (\partial_t \phi + b \cdot
\nabla \phi)$$ 
while the left hand side goes to $-\int_{H^{-1}(K)^c} \beta(u^o)
\phi^o$.

Then, to prove that (\ref{test}) holds, we just have to show that 
\begin{equation} \label{residu}
 \int_{[0,\infty) \times H^{-1}(K)} \beta(u) (\partial_t \phi + b \cdot
\nabla \phi) = -\int_{H^{-1}(K)} \beta(u^o) \phi^o 
\end{equation}
But  $H \in W^{1,p}(\RR^2)$ and $K$ is a zero-measure set, and this is a
classical result that in that case $\nabla H = 0$ a.e. on
$H^{-1}(K)$ (see for instance \cite{Bou}). Then, $b=\nabla
H^{\perp}=0$ a.e. on this set, and
$$ \int_{[0,\infty) \times H^{-1}(K)} \beta(u) (\partial_t \phi + b \cdot
\nabla \phi) =\int_{[0,\infty) \times H^{-1}(K)} \beta(u)\partial_t \phi$$

Moreover, $H^{-1}(K) \cap P$ is a set of zero-measure because on $P$,
$\nabla H \neq 0$ a.e.. Hence the following quantity will not change
if we integrate only on $H^{-1}(K) \cap (O \cup Z)$ or on $H^{-1}(K)
\cap O$ since $Z$ has zero-measure. But we already know that $u$ is
independent of the time on this set, then we can integrate in time to
obtain the equality (\ref{residu}).
\end{proof}

As a conclusion to this section we just wanted to say that we do not
know what happens when the condition $(P_x)$ is not true on a
sufficiently large set. Of course, we can construct divergence free
vector-fields which do not satisfy $(P_x)$ at every point, but it seems
difficult to work with such flows because their definition is complex.

%%%%%%%%%%%%%%%%%%%%%%%%%%%%%%%%%%%%%%%%%%%%%%%%%%%%%%%%%%%%%%%%%
\section{One example}
We observe in this section that the example introduced by R. DiPerna
and P.L. Lions
in \cite{DPL} provides an example of an divergence free vector fields $b$
such that $b \in \Linf_{loc}(\RR^2 \backslash (0,0))$, $b$ is in
$\Lp$ in a neighbourhood of the origin for all $p<2$ but not for
$p=2$, $b$ satisfies the condition $(P_x)$ everywhere, but there
exist several solutions to the transport equations and several
a.e. flows solving the associated ODE. 

\subsection{Definition of the vector-field}
We define the hamiltonian $H$ as follows (see fig \ref{figex})
$$ H(x) = \left\{\begin{array}{ll}
- \frac{x_1}{|x_2|} & \text{if  } |x_1| \leq |x_2| \\
-(x_1 - |x_2| +1) & \text{if  } x_1 > |x_2| \\
-(x_1 + |x_2| -1) & \text{if  } x_1 < -|x_2| 
\end{array} \right. $$
Then, $b$ is given by
$$b_1(x) =-\frac{\partial H}{\partial x_2} =-sign(x_2)\left( {
\frac{x_1}{|x_2|^2} 1_{|x_1| \leq |x_2|} + sign(x_1) 1_{|x_1| >
  |x_2|}} \right) $$
$$ b_2(x) = \frac{\partial H}{\partial x_1}= -\left( {\frac{1}{|x_2|}
1_{|x_1| \leq |x_2|} + 1_{|x_1|>|x_2|}}\right) $$

\begin{figure}[h]
\centering
\includegraphics{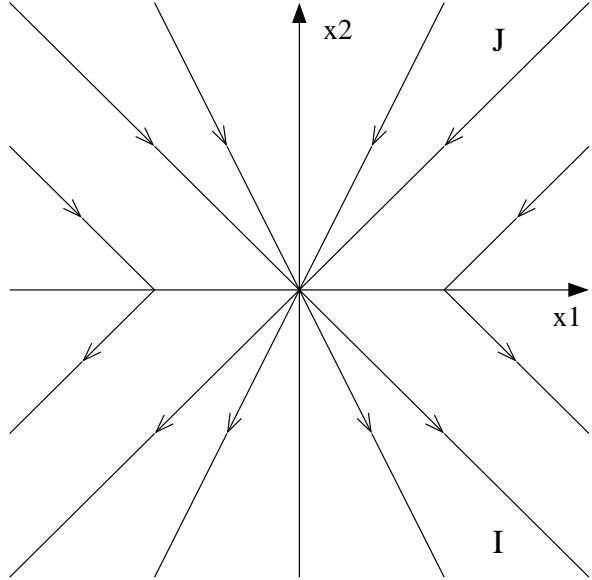}
\caption{ Flow lines of the example} \label{figex}
\end{figure}

\subsection{Form of the solutions}
First, we construct an a.e. flow $X$ solution of the associated ODE.
Since it would be symmetric in relation to the $x_2$-axis, we only
defined it for $x_1 \geq 0$. We also define it just for $t \geq 0$.

In the case when $0 \leq x_2 \leq x_1$, we set
\begin{gather*}
X(t,x)=(x_1 -t,x_2-t) \qquad \text{for} \quad t \leq x_2\\
X(t,x)=(x_1-2x_2 +t,x_2-t) \qquad \text{for} \quad t \geq x_2
\end{gather*}
while for $0 \leq -x_2 \leq x_1$, we set
$$X(t,x)=(x_1+t,x_2 -t)$$
In the case when $0 \leq x_1 < x_2$, we set 
\begin{gather*}
X(t,x)= \sqrt{1- \frac{2t}{(x_2)^2}}\,\,(x_1,x_2) \qquad
\text{for} \quad t \leq \frac{(x_2)^2}{2} \\
X(t,x)= \sqrt{\frac{2t}{(x_2)^2}-1}\,\,(x_1,-x_2) \qquad
\text{for} \quad t \geq \frac{(x_2)^2}{2} \\
\end{gather*}
And if $0 \leq x_1 < -x_2$,
$$ X(t,x)=\sqrt{1+\frac{2t}{(x_2)^2}}\,\,(x_1,x_2) $$

In the sequels, we denote $I=\{x \in \RR^2 | 0 < x_1 < -x_2\}$ and
$J=\{x \in \RR^2 | 0 < x_1 < x_2\}$. For an initial condition $u^o$,
some tedious computation easily shows that the solutions of the transport
equation (we use the fact that $u(t,X(t,x))$ is independent of $t$ as
long as X(t,x) does not reach the origin, and then we
use the change of variable $(t,x) \rightarrow (t,X(t,x))$ on all the
space, paying attention to what happens at the origin). They are of
the form

\begin{equation}
u(t,x) = \left\{ \begin{array}{ll}
u^o(X(-t,x)) & \text{if} \quad  x \notin I \quad \text{or}\quad x \in I
\quad \text{and} \quad t \leq \frac{(x_2)^2}{2} \\
\tilde{u}(X(-t,x)) & \text{if} \quad x \in I \quad \text{and} \quad t
\geq \frac{(x_2)^2}{2} \end{array} \right.
\end{equation}
where $\tilde{u}$ is any function defined on $J$ satisfying the
condition
\begin{equation} \label{masscons}
\forall x_2 > 0, \quad \int_{-x_2}^{x_2}\tilde{u}(x_1,x_2)\,dx_1 =
\int_{-x_2}^{x_2}u^o(x_1,x_2)\,dx_1
\end{equation}

Indeed, we use here the flow $X$ for simplicity but these solutions are
 not defined according to this flow when a trajectory pass through the
 origin. We will try to explain what happens at the origin. For $x_2 >
 0$, if the quantity $u$ represent a density of mass, all the mass on
 the segment $\{(x,x_2)|x \in (-x_2,x_2)\}$
 reaches the origin at the time $(x_2)^2/2$. After this time it
 continues to move in $I$ always on segments parallel to the
 $x_1$-axis, but it can be redistributed on them in any way
 provided the total mass on this segment in conserved. This is what
 means the condition (\ref{masscons}).

The renormalized solutions are always of this form, but the condition
(\ref{masscons}) should be replaced by 
\begin{equation} \label{masscons2}
\forall x_2 > 0, \forall \beta \in \CCC^1(\RR) \quad
\int_{-x_2}^{x_2}\beta(\tilde{u})(x_1,x_2)\,dx_1 =
\int_{-x_2}^{x_2}\beta(u^o)(x_1,x_2)\,dx_1 
\end{equation}
This condition (\ref{masscons2}) is equivalent to the fact that for
all $x_2 \in \RR$, we have a measure-preserving transformation $\Phi$
from $(-x_2,x_2)$ into itself such that $\tilde{u}=u^o \circ \Phi$. We
refer to \cite{Roy} for this point.

Moreover, we can also find all the flows solutions of the associated
ODE. Choosing a mesurable measure-preserving transformation $\Psi$
from $(-1,1)$ into itself, we defined a flow $X_{\Psi}$ by 
\begin{equation*}
X_{\Psi}(t,x)= \left\{ \begin{array}{ll} 
X(t,x) & \text{if} \quad x \notin J \quad \text{or} \quad t \leq
\frac{(x_2)^2}{2} \\
\Psi(x) X(t,x) & \text{if} \quad x \in J \quad \text{and} \quad t \geq
\frac{(x_2)^2}{2} \end{array} \right.
\end{equation*}

To see that this defined a a.e. flow, we use the property stated in
the definition of an a.e. flow and the fact that an a.e. flow is
measure preserving. Let us try to illustrate this definition. A
particle with an initial position $x^o$ in $J$ moving according to
$X_{\Psi}$ behaves as follow. It moves on the half-line $\{x|x_1/x_2 =
\lambda,\quad x_2> 0\}$ (with $\lambda=x_1^o/x_2^o$) until it reaches
the origin. Then it continues to move in $I$ but on the half-line
$\{x|x_1/x_2 = \Psi(\lambda),\quad x_2 < 0\}$. Indeed, $\Psi$ may be
seen as a mapping between the upper half-lines and the lower ones.

We can thus see that in this case, we have renormalized solutions that
are not defined according to an a.e. flow. Indeed, for a renormalized
solution, we can choose different mappings between the
upper half-lines and the lower ones for each $x_2$ (in other words $\tilde{u} =
u^o(\Psi_{x_2} (x_1/x_2)x_2,x_1)$ where $\Psi_{x_2}$ is measure-preserving
transformation from $(-1,1)$ into itself depending on $x_2$),
while for a solution defined according to an a.e. flow, this
correspondance will be independant of $x_2$ .

\subsection{Remark about the uniqueness of the solution}
First we remark that the flow $X$ is a specific one. It is the only one for
which the hamiltonian remains constant on all the trajectories.
Moreover, we observe that the solution $\overline{u}$ defined according
to $X$ is specific among all the others. This is the only one which
satisfies also the above family of equations (\ref{hamilt}), which says
that the hamiltonian remains constant on the trajectories.
\begin{equation} \label{hamilt}
 \forall f\in \CCC(\RR,\RR) \qquad \partial_t(f(H) u)+\diver(f(H)bu) = 0
\end{equation}

Indeed, we do the same computation that leads to (\ref{masscons}) with these
equations and we obtain the following conditions
\begin{equation}
\forall x_2 > 0,\, f \in \CCC(\RR,\RR), \quad 
\int_{-x_2}^{x_2}\tilde{u}(x_1,x_2) f(x_1) \,dx_1 =
\int_{-x_2}^{x_2}u^o(x_1,x_2) f(x_1) \,dx_1
\end{equation}
This implies that $\tilde{u}=u^o$ and then that $u=\overline{u}$.

Hence, adding the conditions (\ref{hamilt}) in the
definition of a solution, we are able to define it uniquely. Moreover,
if we try to solve this problem by approximation, choosing a sequence
of divergence free vector-field $b_m$ converging to $b$ in all
$\Lp_{loc}$, for $p<2$, (this implies that $H_m$ converge to $H$ up to
a constant in all $W^{1,p}_{loc}$, for $p<2$), we obtain a sequence of
solutions $u_m$, that satisfy all the equations (\ref{hamilt}) with
the initial condition $u^o$. This sequence $u_m$ is weakly compact in
$\Linf$. Extracting a converging subsequence, we see that the
equations (\ref{hamilt}) are always true at the limit and then the
sequence $u_m$ converge to $\overline{u}$. This solution is therefore
the only that we can construct by approximation.

%%%%%%%%%%%%%%%%%%%%%%%%%%%%%%%%%%%%%%%%%%%%%%%%%%%%%%%%%%%%%%%%
\section{The case of a particule moving on a line}

We consider here a classical Hamiltonian
$$H(x,y)=y^2/2+ V(x) \quad \text{or} \quad b(x,y)=(y, -V'(x))$$
with $V$ a potential in $W^{1,p}_{loc}(\RR)$. Then $V'$ belongs to
$\Lp_{loc}(\RR)$. In this case, $b$ satisfies the $(P_x)$ assumption
in $\RR^2 \backslash \{y=0\}$. The set of criticals points $Z$ has
then zero-measure. We can apply the preceding results, if $H$ satisfies
$m(H(Z))=0$. When $V'$ is continuous, this is true because we can apply
the Sard Lemma. But this is false for a general $V' \in \Lun_{loc}$. If
$V'$ oscillates very quickly, $Z$ may even be the whole line. And then
$H(Z)$ is an interval because $H$ is continuous. However, we will show
that the result is always true in this case. Moreover, we can only
assume that $V'$ belongs to $\Lun_{loc}(\RR)$, since the other
composant of $b$ is in $\Linf_{loc}(\RR)$. 

The transport equation we are considering has the form
\begin{equation} \label{transham}
\frac{\partial u}{\partial t} + y \frac{\partial u}{\partial x}
-V'(x)\frac{\partial u}{\partial y} =0
\end{equation}

Here we can solve the differential equation $x'=y$, $y'=-a(x)$ directly
if we use the fact that the Hamiltonian is constant on a trajectory
and integrate the system. But this flow is not regular, and we do not
know how to work directly with it in order to solve the transport equation.  

\begin{thm}
For a flow $b(x,y)=(y, -V'(x))$ satisfying $V' \in \Lun_{loc}(\RR)$ and also $
[\max (1,-V(x))]^{-1/2}$ not integrable at $\pm \infty$, the
transport equation has an unique renormalized solution. 
\end{thm}

\textbf {Remarks}
\begin{itemize}
\item[i.]  The condition of integrability on $V$ is there to insure
  that a point  does not reach $\pm \infty$ in a finite time. It could
  be replaced by a stronger condition like $V(x) \geq -C(1+x^2)$.
\item[ii.] This result can be adapted to the case of two particles
  moving on a line according to a interaction potential in
  $W^{1,1}_{loc}$. In order to do so we just have to use a change
  of variable which follows the classical way of reducing this
  two-body problem to a one-body problem.
\end{itemize}

\begin{proof}
We can use our previous theorem in the neighboorhood
of a point with $y \neq 0$. This will give us ``the result'' on two
half-planes, but we need to ``glue'' together the information available
on this two half-planes. Then we need to work differently, and we shall
follow the same sketch of proof as in our first theorem.

\textbf{Step 1.} A change of variables. \\
First we define 
$$\Phi_+ (x,y)= (x,y^2/2 + V(x)) \qquad \text{from} \quad \RR\times
(0,\infty) \quad \text{to} \quad B $$
$$\Phi_- (x,y)= (x,y^2/2 + V(x)) \qquad \text{from} \quad \RR\times
(-\infty,0) \quad \text{to} \quad B $$
where $B = \{(x,E) \in \RR^2| V(x) < E \}$.
Then $\Phi_+$ and $\Psi_-$ are continuous and belong to $W^{1,1}_{loc}$ with
$$ D\Phi_{\pm} = 
\left(\begin{array}{cc}
1 & 0 \\
V'(x) & y \\
\end{array} \right)$$
and the same for $\Phi_-$.

 These transformations are one-to-one and onto
and 
$$\Phi_{\pm}^{-1}(x,E)= (x, \pm \sqrt{2(E-V(x)})$$.
$$ D\Phi_+^{-1} = 
\left(\begin{array}{cc}
1 & 0 \\
- \frac{V'(x)}{\sqrt{2(E-V(x))}} & \frac{1}{\sqrt{2(E-V(x))}} \\
\end{array}\right) ,$$
with a similar formula for $\Phi_-^{-1}$.

The following change of variables is true
$$\int_{y>0} f(x,y) \,dxdy = \int_B \frac{f\circ
  \Phi_+^{-1}(x,E)}{\sqrt{2(E-V(x))}} \,dxdE $$
for $f$ in $\Linf$ and even in $\Lun$.

Before going further, we state some properties about the set
$W^{1,1}(B)$. We define $\CCC^{\infty}(\overline{B})$
(resp.$\Com(\overline{B})$)  the space of
restriction to B of $\CCC^{\infty}$-functions on $\RR^2$ (resp. such
functions with compact support). We recall that $\partial B=
\{(x,V(x))| x \in \RR\}$.

\begin{figure}[ht]
\centering
\includegraphics{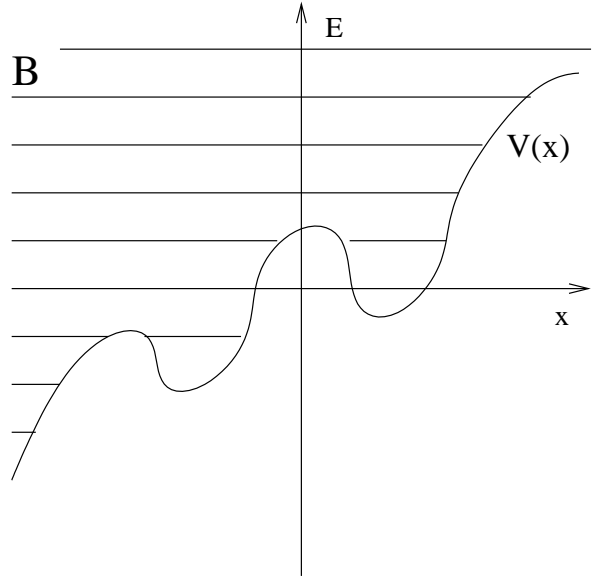}
\caption{ The domain $B$}
\end{figure}

\begin{prop}
\begin{itemize}
\item[]
\item[i.]$\Com(\overline{B})$ is dense in $W^{1,1}(B)$.
\item[ii.]The trace of a function in $W^{1,1}(B)$ has a sense in $\Lun$. More
precisely, there exist a continuous application $Tr: W^{1,1}(B) \rightarrow
\Lun(\RR)$ such that $Tr(\phi)=\phi(\cdot,V(\cdot))$ if $\phi \in
\Com(\overline{B})$. 
\item[iii.]$\overline{\Com(B)}= Ker(Tr)$, the kernel of the
  trace, also denoted by $W^{1,1}_o(B)$.
\item[iv.] The same results are true for $\RR \times B$ and $[0,\infty)
  \times B$ as well as locally.
\end{itemize}
\end{prop}

\begin{proof}[Proof of the proposition]
We refer to theorem 3.18 in \cite{Ada} for a complete proof. But we may
adapt the proof of this theorem to this simpler case. For a $f \in
W^{1,1}(B)$, we define for $\epsilon >0$ 
$$f_{\epsilon}= \rho_{\epsilon} * (f(\cdot + 2\epsilon e)\chi_{\{E >
  V(x) - 2\epsilon\}})$$
where $e=(0,1)$ and $\rho_{\epsilon}(x,E)= \rho(x
/\epsilon,E/\epsilon)$ with $\rho \in\Com(\RR^2)$ satisfying $\int
\rho =1$. Then, the functions $f_\epsilon$ belong to
$\Com(\overline{B})$ and converges to $f$ in $W^{1,1}(B)$ as $\epsilon
\rightarrow 0$.

For the second point, we choose $f \in \Com(\overline{B})$. Then
$$ f(x,V(x))= -\int_{V(x)}^{\infty} \frac{\partial f}{\partial
  E}(x,E)\,dE $$
taking the absolute value and integrating in $x$ leads to
$$ \int_{\RR} |f(\cdot,V(\cdot))| \leq \|\nabla f\|_{\Lun(B)} $$

Then, the trace is a contraction from $\Com(\overline{B})$ with the
$W^{1,1}$-norm into $\Lun(\RR)$, and since $\Com(\overline{B})$ is dense
in $W^{1,1}(B)$, we may extend this application to
$W^{1,1}(B)$.

For the third point, we take $f \in Ker(Tr)$ and extend it by zero
outside $B$. We obtain a $\tilde{f}$ in $W^{1,1}(\RR^2)$. Then if
we translate $\tilde{f}$ in the direction of $e=(1,0)$ and smooth it
by convolution, we can construct $\CCC^{\infty}$-approximations of $f$
with support in $B$.  
\end{proof}

\textbf{Step 2.} Equivalence with a simpler transport equation.

Now, let $u$ be a solution of the transport equation. We may write
\begin{equation} \label{ham}
 \int_{[0,\infty) \times \RR^2} u (\partial_t \phi + y \partial_x \phi - V'(x)
\partial_y \phi) \,dxdy = -\int_{\RR^2} u^o \phi^o 
\end{equation}\
for all $\phi \in W^{1,1}([0,\infty) \times \RR^2)$ with compact
support (in the sense of  distributions) satisfying moreover
$\partial_y \phi \in \Linf([0,\infty) \times \RR^2)$.

Let $\Psi_+$ and $\Psi_-$ be in $\Com([0,\infty) \times
\overline{B})$, and $\Psi_+$ and $\Psi_-$ satisfy the compatibility
condition ${\Psi_+}_{|[0,\infty) \times \partial B}
={\Psi_-}_{|[0,\infty) \times \partial B}$. We define $\phi$ from
$[0,\infty) \times \RR^2$ to $\RR$ with
$$ \phi(t,x,y)= \left\{ \begin{array}{ll} \Psi_+ (t, \Phi_+(x,y)) & \text{if}
  \quad y>0 \\ 
\Psi_- (t, \Phi_-(x,y)) & \text{if}  \quad y<0  \end{array} \right. $$
Then, $\phi$ belongs to $W^{1,1}([0,\infty) \times \RR^2)$, has a compact
support, and $\phi$, $\partial_t \phi$, $\partial_y \phi$ are in
$\Linf$.
Moreover,
\begin{gather*}
 \partial_x \phi = (\partial_x \Psi_+) \circ \Phi_+ + E(x)
(\partial_y \Psi_+) \circ \Phi_+ \text{\ \ \ for } y>0  \\
\partial_y \phi = y (\partial_y \Psi_+) \circ \Phi_+    \\
\partial_t \phi = (\partial_t \Psi_+) \circ \Phi_+  
\end{gather*}
Then we have \hspace{1cm} $\partial_t \phi + y \partial_x \phi -
E(x)\partial_y \phi =(\partial_t \Psi_+) \circ \Phi_+ + y  (\partial_x
\Psi_+) \circ \Phi_+$ for all $y>0$.

We write $v_{\pm}= u \circ \Phi_{\pm}^{-1}$, defined on $[0,\infty) \times
B$. Then, (\ref{ham}) may be written as follows.
\begin{multline}
\int_{[0,\infty) \times \{y>0\}} [v_+ (\partial_t \Psi_+ +
\sqrt{2(E-V(x))}\partial_x \Psi_+)] \circ \Phi_+ \\
 + \int_{[0,\infty)   \times \{y<0\}} [v_- (\partial_t \Psi_- -
 \sqrt{2(E-V(x))}\partial_x\Psi_-)] \circ \Phi_- \\ 
= \int_{y>0} (v_+^0 \Psi_+^0) \circ \Phi_+ +\int_{y<0} (v_-^0
\Psi_-^0) \circ \Phi_-   
\end{multline}
We can apply the change of variables, and we obtain
\begin{multline} \label{hamch}
\int_{[0,\infty) \times B} v_+ ( \frac{\partial_t
  \Psi_+}{\sqrt{2(E-V(x)})} + \partial_x \Psi_+) + \\ 
  \int_{[0,\infty)
  \times B} v_- ( \frac{\partial_t \Psi_-}{\sqrt{2(E-V(x))}} -
  \partial_x \Psi_-)
 = \int_B \frac{v_+^0 \Psi_+^0 + v_-^0
  \Psi_-^0}{\sqrt{2(E-V(x))}} 
\end{multline}
It is difficult to work with $\Phi_+$ and $\Phi_-$ because of the
compatibility condition. But we may make the particular choice
$\Phi_+=\Phi_-$ (below we will omit the indices $\pm$). Then
(\ref{hamch}) becomes 
\begin{equation} \label{hamcheven}
\int_{[0,\infty) \times B} (v_+ + v_-)\frac{\partial_t
  \Psi}{\sqrt{2(E-V(x))}} + (v_+ -v_-) \partial_x \Psi = \int_B
  \frac{(v_+^0 + v_-^0}{\sqrt{2(E-V(x))}} \Psi^0  
\end{equation}
Now, we choose $\Psi_+=-\Psi_-$ and $\Psi_{|[0,\infty) \times
  \partial B}=0$ (we omit the indices $\pm$). In this case,
  (\ref{hamch}) becomes  
\begin{equation} \label{hamchodd}
\int_{[0,\infty) \times B} (v_+ - v_-) \frac{\partial_t
  \Psi}{\sqrt{2(E-V(x))}} + (v_+ + v_-) \partial_x \Psi = \int_B
  \frac{(v_+^0 - v_-^0)}{\sqrt{2(E-V(x))}} \Psi^0  
\end{equation}

Then, (\ref{ham}) implies (\ref{hamcheven}) for all $\Psi$ in
$\Com([0,\infty) \times \overline{B})$, and (\ref{hamchodd}) for all
$\Psi \in \Com([0,\infty) \times \overline{B})$ with $\Psi_{|[0,\infty)
  \times \partial B}=0$, or equivalently for all $\Psi \in
\Com([0,\infty) \times B)$ since $\Com([0,\infty) \times B)$ is dense
in $W^{1,1}_o([0,\infty) \times B)$. And conversly, these two statements are equivalent with
(\ref{hamch}) for all $\Psi_+$ and $\Psi_-$ in $\Com([0,\infty) \times
\overline{B})$ having the same trace on the boundary.
 
Thus, we have to solve
\begin{equation} \label{cov1}
\partial_t( \frac{v_+ + v_-}{\sqrt{2(E-V(x))}}) + \partial_x(v_+ -v_-)
  =0 \qquad \text{on} \quad \DDD'([0,\infty) \times \overline{B})
\end{equation}
\begin{equation} \label{cov2}
\partial_t(\frac{v_+ - v_-}{\sqrt{2(E-V(x))}}) + \partial_x(v_+ +v_-)
  =0 \qquad \text{on} \quad \DDD'([0,\infty) \times B)
\end{equation}
with the convenient initial conditions. In (\ref{cov1}),
$\DDD'([0,\infty) \times \overline{B})$ means that we allow test functions
in $\CCC^{\infty}([0,\infty) \times \overline{B})$.

We can do the same arguments backwards. Therefore, solving
(\ref{cov1})-(\ref{cov2}) is equivalent to solve (\ref{transham}) 

\textbf{Step 3.} Reduction to one dimension.\\
 These two equations do not contain any derivative in $E$. As in the
 proof of the first result, we want to  reduce them to equations in
 one dimension of space. For the second  equation (\ref{cov2}), we can
 make the same argument and we obtain that this equation holds on $B_E$, for
 almost all $E$ in $\RR$. (with $B_E=\{x \in  \RR | (x,E) \in B\}$).

For the first equation (\ref{cov1}) we can still apply the
argument. We shall be more precise since it is a little bit more
involved. We choose a test function $\phi$ of the form $\phi_1 \phi_2$
with $\phi_1$ depending only on (t,x) and $\phi_2$ depending on
$E$. We obtain 
\begin{multline} \label{product}
\int_{[0,\infty) \times B} \left({(v_+ + v_-)\frac{\partial_t
\phi_1}{\sqrt{2(E-V(x))}} +  (v_+ -v_-) \partial_x \phi_1}\right)
\phi_2  \\
= \int_B   \frac{(v_+^0 + v_-^0)}{\sqrt{2(E-V(x))}}
\phi_1^o \phi_2
\end{multline}

Since the linear combinaisons of functions of the form $\phi_1 \phi_2$
are dense in $\Com([0,\infty))$ with the $W^{1,1}$-norm,
(\ref{product}) for all $\Com$ $\phi_1$ and $\phi_2$ is equivalent
with (\ref{hamcheven}) for all $\Com$ $\Psi$. Moreover, since
 $W^{1,1}([0,\infty) \times \RR)$ is separable, it
is sufficient (and necessary) to write (\ref{product}) for $\phi_1$
choosen among a countable subset $F_1$ of $\Com$-functions. 

Now, using Fubini's theorem (\ref{product}) may be rewritten 

\begin{multline} \label{fubini}
\int_{\RR} \left({ \int_{[0,\infty) \times B_E} (v_+ + v_-)\frac{\partial_t
  \phi_1}{\sqrt{2(E-V(x))}} + (v_+ -v_-) \partial_x \phi_1 \,dtdx}\right)
\phi_2 \, dE \\
= \int_{\RR} \left( { \int_{B_E} \frac{(v_+^0 +
      v_-^0)}{\sqrt{2(E-V(x))}} \phi_1^o \,dx} \right)\phi_2 \,dE
\end{multline}
for a fixed $\phi_1$. Since it is satisfied for all $\Com$-$\phi_2$,
we obtain that

\begin{equation}
 \int_{[0,\infty) \times B_E} (v_+ + v_-)\frac{\partial_t
  \phi}{\sqrt{2(E-V(x))}} + (v_+ -v_-) \partial_x \phi = \int_{B_E}
  \frac{(v_+^0 + v_-^0)}{\sqrt{2(E-V(x))}} \phi^o
\end{equation}
for all $E \in \RR \backslash N$ where $N$ is a zero-measure set
depending on $\phi_1$. Now, if we write this equation for all
$\phi_1 \in F_1$, we obtain that (\ref{cov1}) is satisfied, but this time in
$[0,\infty) \times B_E$ for allmost all $E \in \RR$. And we can do the
argument backwards to show that this is equivalent to the initial
problem. Finally, we just have to solve (\ref{cov1})-(\ref{cov2}) on
$B_E$ instead of $B$. 

\textbf{Step 4.} Solution of the one dimensionnal problem.

$B_E$ is a countable union of disjoint open intervals. We denote $B_E=
\cup_n (a_n,b_n)$, where $a_n$, $b_n$ are disjoints reals. But,
since we shall also work on $\overline{B_E}$ we
want that these open intervals are not to ``close'' to each other. For
instance, if there exist $n, \; m$ such that $b_n=a_m$ and if
$1/\sqrt{2(E-V(x))}$ is integrable on a neighboorhood of $b_n$, a
particle reaching $b_n$ from the left may continue to go further right or may
change direction and go backwards. This will give rise to distinct
solutions of the transport equation. But, we shall show that for almost
all $E$, we have some ``free zone'' around each $(a_n,b_n)$. More
precisely, for almost all $E$ there exists an $\epsilon_n > 0$
such that $B_E \cap (a_n - \epsilon_n, b_n + \epsilon_n) = (a_n,
b_n)$. If we admit this point, we see that we just have to solve
(\ref{cov1})-(\ref{cov2}) on an interval of the type $(a',b')$, where $a'$
belongs to $[-\infty,+\infty)$ and $b'$ to $(-\infty,+\infty]$. Before
going further, we prove the

\begin{lemma}
For almost all $E$, if we write $B_E=\cup_n (a_n,b_n)$ then, for each
$n$, there exists some $\epsilon_n > 0$ such that $B_E \cap
(a_n - \epsilon_n, b_n + \epsilon_n) = (a_n,b_n)$
\end{lemma}

\begin{proof}[Proof of the lemma]
First we recall that since $V$ belongs to $W^{1,1}_{loc}$, the image by
$V$ of a zero-measure set is a zero-measure set.
Then, we state a result similar to the Sard's lemma for V. Let  $Z$ be the
set were $V'$ vanishes. We claim that $V(Z)$ has zero-measure. Of
course , $Z$ is defined up to a zero-measure set, but this is
irrelevant for our claim in view of the fact recalled above.
In order to prove our claim, we choose a sequence of open sets $O_n$
such that $Z \subset  O_n$ and $\lambda(O_n \backslash Z)$ goes to $0$
as $n$ goes to $\infty$. Here and below $\lambda$ denotes the Lesbegue measure
on $\RR$ or $\RR^2$. We may write $O_n = \cup_m I_{n,m}$ where the
$I_{n,m}$ are disjoint intervals of $\RR$. Then,
\begin{eqnarray*}
\lambda(V(O_n)) & = & \lambda(V(\cup_m I_{n,m})) \leq \sum_m
\lambda(V(I_{n,m}))  \\
 & \leq  & \sum_m \int_{I_{n,m}} |V'| = \int_{O_n} |V'|  \\
 & \leq &  \int_{O_n \backslash Z} |V'|
\end{eqnarray*}
and the last quantity goes to $0$ as $n \rightarrow \infty$ since
$\lambda(O_n \backslash Z)$ goes to $0$ as $n \rightarrow \infty$ and
our claim is shown.

Next, we denote by $Z_1$ the set such that $Z_1^c$ is the set of
Lebesgue points of $V'$ (i.e. the set of points such that
$1/(2\epsilon) \int_{x-\epsilon}^{x+\epsilon} |V'(y)-V'(x)| \,dy$ goes
to zero as $\epsilon \rightarrow 0$).  Then,
$\lambda(Z_1)=0$. According to what we proved above, we know that
$\lambda(V(Z \cup Z_1))=0$. Now, if we choose $E \in V(Z \cup
Z_1)^c$,and write $B_E=\cup_n (a_n,b_n)$ as above, we know that $a_n$
and $b_n$ are Lebesgue's point of $V'$ with $V'(a_n)\neq 0$ and
$V'(b_n)\neq 0$. Then necessarily, $V'(b_n) > 0$ and $V$ is strictly
increasing in a neighboorhood of $b_n$ because it is a Lebesgue's
point. Since we may make the same argument near $a_n$, we have then
shown the existence of $\epsilon_n$ as stated in the lemma.  
\end{proof}

To solve (\ref{cov1})-(\ref{cov2}) on $(a',b')$ we use the change of
variable $x \mapsto z= F(x)$  where $F$ is a primitive of
$1/\sqrt{2(E-V(x))}$ from $(a',b')$ to $(a,b)$. We can because this
quantity is locally integrable on almost all lines (this result is
easily seen using Fubini's theorem). Then, we obtain the two following
equations   
\begin{equation} \label{ncov1}
\partial_t(w_+ + w_-) + \partial_z(w_+ -w_-) =0 \qquad  \text{on} \quad
[0,\infty) \times [a,b] 
\end{equation}
\begin{equation} \label{ncov2}
\partial_t(w_+ - w_-) + \partial_z(w_+ +w_-)
  =0 \qquad  \text{on} \quad [0,\infty) \times (a,b)
\end{equation}
with appropriate initial conditions. And as before, in (\ref{ncov1})
we use test functions in $\Com([0,\infty) \times [a,b])$ (in others
words the tests functions do not necessarily vanish on $\{z=a\}$ and
$\{z=b\}$ when $a$ and $b$ are finite).

Here, if $a'=-\infty$ or $b'=+\infty$ we need the assumption of
non-integrability on $V$. If it is not verified, $a$ (or $b$) will
be finite, and we cannot use test functions which do
not  vanish on $\{z=a\}$ (or $\{z=b\}$) in (\ref{cov1}). And we shall not have
the uniqueness of solutions of the equivalent problem (as will become
clearerr below).

 Adding and substracting the two equations in $\DDD'([0,\infty) \times
(a,b))$ yields
$$ \partial_t w_+ + \partial_z w_+ = 0 \qquad \text{in} \quad
\DDD'([0,\infty) \times (a,b)) $$
$$ \partial_t w_- - \partial_z w_- = 0 \qquad \text{in} \quad
\DDD'([0,\infty) \times (a,b)) $$

Hence, the solutions are of the form $w_+(t,z)=\Phi_+(z-t)$ and
$w_-(t,z)=\Phi_-(z+t)$ with $\Phi_+$ and $\Phi_-$ belonging to
$\Linf(\RR)$. but we have not used yet the fact that (\ref{ncov1}) is
true on [a,b]. This tells us formally that $w_+(t,a)=w_-(t,a)$ when $a
\neq -\infty$ and $w_+(t,b)=w_-(t,b)$ when $b \neq +\infty$. This can
be justified. Indeed, let us assume that $b \neq +\infty$ and let we
choose some $\phi \in \Com((0,\infty))$, an $\epsilon \in (0, b-a)$ and
$\chi_{\epsilon} \in \CCC^{\infty}(\RR)$ increasing such that
$\chi_{\epsilon}(z)=0$ for $z < b-\epsilon$ and some $\chi_{\epsilon}(z)=1$
for $z> b$ . We use $\phi \chi_{\epsilon}$ as a test function in
(\ref{ncov1}). We then obtain
\begin{multline*}
\int_{[0,\infty) \times (b -\epsilon , b)} (\Phi_+(z-t) +
\Phi_-(z+t))\partial_t \phi(t) \chi_{\epsilon}(z)\,dtdz  \\
+ \int_{[0,\infty) \times (b -\epsilon , b)} (\Phi_+(z-t)
-\Phi_-(z+t)) \phi(t)\partial_z \chi_{\epsilon}(z)\,dtdz = 0
\end{multline*}
When $\epsilon \rightarrow 0$, the first integral goes to $0$. The
second integral goes to $\int_{[0,\infty)} (\Phi_+(b-t) -
\Phi_-(b+t))\phi(t) \,dt$. Since it holds for all $\phi \in
\Com((0,\infty))$, we obtain that $\Phi_+(b-t) =\Phi_-(b+t)$. We can
prove similary that $\Phi_+(a-t) =\Phi_-(t+a)$ if $a \neq -\infty$. 

Now, we shall assume that $a$ and $b$ are both finite (the other cases
are similar and simpler) and we define $l=b-a$. Without using the
boundary conditions, the initial conditions
on $w_+$ and $w_-$ impose the value of $\Phi_+$ and $\Phi_-$ on the
interval $(a,b)$. Of course, it should be understood in sense of
%revoir ici ae
functions defined almost everywhere, but here it does not raise any
difficulty and we will omit to specify it afterwards. Using the
boundary condition $\Phi_+(b-t) =\Phi_-(b+t)$, we see that $\Phi_+$ and
$\Phi_-$ are determined in $(b,b+l)$.  And the condition $\Phi_+(a-t)
=\Phi_-(t+a)$ determines $\Phi_+$ and $\Phi_-$ in $(a-l,a)$. If we
continue to use this symmetry argument further, we see that $\Phi_+$
and $\Phi_-$ are uniquely determined in $\RR$, provided we know them
in $(a,b)$ (we remark here that it is not the case if one of the
boundary counditions is missing, as it is the case when $a=-\infty$ or
$b=+\infty$ and  the assumption of non-integrability on $V$ is not
satisfied).Then, for every intial condition (on $w_+$ and $w_-$) in
$\Linf$, there exists a unique solution to the system
(\ref{ncov1})-(\ref{ncov2}). And in view of the form of those
solutions, we see that they are renormalized ones. This concludes the
proof.

\begin{figure}[h]
\centering
\includegraphics{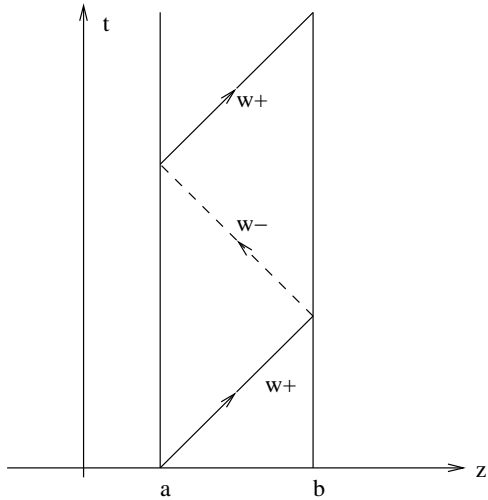}
\caption{ behaviour of $w_+$ and $w_-$} \label{bound}
\end{figure}

\end{proof}

\def\cprime{$'$} \def\cprime{$'$}


\begin{thebibliography}{1}

\bibitem{Ada}
R.~A. Adams.
\newblock {\em Sobolev spaces}.
\newblock Academic Press [A subsidiary of Harcourt Brace Jovanovich,
  Publishers], New York-London, 1975.
\newblock Pure and Applied Mathematics, Vol. 65.

\bibitem{Bou}
F.~Bouchut.
\newblock Renormalized solutions to the {V}lasov equation with coefficients of
  bounded variation.
\newblock {\em Arch. Ration. Mech. Anal.}, 157(1):75--90, 2001.

\bibitem{BD}
F.~Bouchut and L.~Desvillettes.
\newblock On two-dimensional {H}amiltonian transport equations with continuous
  coefficients.
\newblock {\em Differential Integral Equations}, 14(8):1015--1024, 2001.

\bibitem{DPL}
R.~J. DiPerna and P.-L. Lions.
\newblock Ordinary differential equations, transport theory and {S}obolev
  spaces.
\newblock {\em Invent. Math.}, 98(3):511--547, 1989.

\bibitem{Lio}
P.-L. Lions.
\newblock Sur les \'equations diff\'erentielles ordinaires et les \'equations
  de transport.
\newblock {\em C. R. Acad. Sci. Paris S\'er. I Math.}, 326(7):833--838, 1998.

\bibitem{Roy}
H.~L. Royden.
\newblock {\em Real analysis}.
\newblock The Macmillan Co., New York, 1963.

\bibitem{Zie}
W.~P. Ziemer.
\newblock {\em Weakly differentiable functions}, volume 120 of {\em Graduate
  Texts in Mathematics}.
\newblock Springer-Verlag, New York, 1989.
\newblock Sobolev spaces and functions of bounded variation.

\end{thebibliography}
\end{document}